\documentclass[leqno,a4paper]{article}
\usepackage{amsmath,amsthm,amssymb}
\usepackage[utf8]{inputenc}

\usepackage{enumerate}
\usepackage{bbm}
\usepackage{todonotes}
\usepackage[hidelinks]{hyperref}

\usepackage{graphicx}

\usepackage{cite}

\newtheorem{Th}{Theorem}
\newtheorem{Prop}{Proposition}
\newtheorem{Lemma}{Lemma}

\theoremstyle{definition}
\newtheorem{Remark}{Remark}
\newtheorem{Def}{Definition}
\newtheorem{Example}{Example}

\newcommand{\beq}{\begin{equation}}
\newcommand{\eeq}{\end{equation}}

\def\scalar(#1,#2){(#1\mid#2)}

\newcommand{\cb}{{\cal B}}

\newcommand{\xbm}{(X,{\cal B},\mu)}

\newcommand{\ot}{\otimes}
\newcommand{\ov}{\overline}

\newcommand{\bs}{\mathbb{S}}

\newcommand{\Q}{\mathbb{Q}}
\newcommand{\R}{{\mathbb{R}}}
\newcommand{\T}{{\mathbb{T}}}
\newcommand{\C}{{\mathbb{C}}}
\newcommand{\Z}{{\mathbb{Z}}}
\newcommand{\N}{{\mathbb{N}}}

\newcommand{\E}{{\mathbb{E}}}
\newcommand{\vep}{\varepsilon}
\newcommand{\va}{\varphi}

\newcommand{\mob}{\boldsymbol{\mu}}
\newcommand{\bnu}{\boldsymbol{\nu}}

\newcommand{\tend}[3][]{\xrightarrow[#2\to#3]{#1}}

\newcommand{\raz}{\mathbbm{1}}
\newcommand{\stack}[2]{\array{c}{\scriptstyle #1}\\[-1.1ex]{\scriptstyle #2}\endarray}
\newcommand{\ir}{\text{irr}}

\title{Automorphisms with quasi-discrete spectrum, multiplicative functions and average orthogonality along short intervals}

\author{H. El Abdalaoui, M. Lema\'nczyk\thanks{Research supported by Narodowe Centrum Nauki grant UMO-2014/15/B/ST1/03736.}, T. de la Rue}

\begin{document}

\maketitle \normalsize

\thispagestyle{empty}
\begin{abstract}We show that Sarnak's conjecture on M\"obius disjointness holds in every uniquely ergodic model
of a quasi-discrete spectrum automorphism. A consequence of this result is that, for each non constant polynomial $P\in\R[x]$ with irrational leading coefficient and for each multiplicative function $\bnu:\N\to\C$, $|\bnu|\leq1$, we have
\[
    \frac{1}{M} \sum_{M\le m<2M} \frac{1}{H} \left| \sum_{m\le n<m+H} e^{2\pi iP(n)}\bnu(n)  \right|\longrightarrow  0
  \]
  as $M\to\infty$, $H\to\infty$, $H/M\to 0$.
\end{abstract}

\section{Introduction} The central result of this paper is the following.

\begin{Th} \label{tw:sa1} Let $S$ be a uniquely ergodic continuous map defined on a compact metric space $Y$. Let $\nu$ be the unique $S$-invariant (ergodic) Borel probability measure on $Y$. Assume that the measure-theoretic system $(Y,\nu,S)$ has quasi-discrete spectrum. Then, for each multiplicative function $\bnu:\N\to\C$, $|\bnu|\leq 1$, each $f\in C(Y)$ with $\int_Y f\,d\nu=0$ and each $y\in Y$, we have
\beq\label{eq:sa1}
\lim_{N\to\infty}\frac1N\sum_{n\leq N}f(S^ny) \bnu(n)=0.\eeq
\end{Th}
If, moreover, we know that $\lim_{N\to\infty}\frac1N\sum_{n\leq N} \bnu(n)=0$, then~\eqref{eq:sa1} holds regardless of the value of $\int_Y f\,d\nu$. In particular, by taking $\bnu=\mob$, the classical M\"obius function: $\mob(1)=1$, $\mob(p_1\ldots p_k)=(-1)^k$ for different prime numbers $p_1,\ldots,p_k$ and $\mob(n)=0$ for other values of $n\in\N$, we obtain the validity of  Sarnak's conjecture \cite{Sa} in the class of systems under consideration. Recall that Sarnak's conjecture states that, if $T$ is a continuous map of a compact metric space $X$ with zero topological entropy,  then $\frac1N\sum_{n\leq N}f(T^nx)\mob(n)\to 0$ for each $f\in C(X)$ and $x\in X$.
Sarnak's conjecture has been proved to hold in several cases~\cite{Ab-Ku-Le-Ru,Ab-Le-Ru,Ab-Ka-Le,Bo0,Bo,Bo-Sa-Zi,Do-Ka,Fe-Ma,Gr,Gr-Ta,Kar,Li-Sa,Ma-Ri1,Ma-Ri2,Pe,Sa,Ve}.
Automorphisms with quasi-discrete spectrum have zero measure-theoretic entropy, hence the topological entropy will be equal to zero in each uniquely ergodic model of them.

%

So far, the only automorphisms for which Sarnak's conjecture was known to hold
in each of their uniquely ergodic model were those whose prime powers are disjoint in the sense of Furstenberg \cite{Fu}. This assertion is already present in \cite{Bo-Sa-Zi}, see also \cite{Ab-Le-Ru}. However, it is easy to see that even irrational rotations do not enjoy this strong disjointness property (indeed, powers have ``large'' common factors). But Theorem~\ref{tw:sa1} is known to hold for them \cite{Bo-Sa-Zi}. On the other hand, irrational rotations have many other uniquely ergodic models. For those in which eigenfunctions are continuous, e.g.\ symbolic models like Sturmian models \cite{Fo},  Sarnak's conjecture has also been proved to hold in view of a general criterion on lifting Sarnak's conjecture by continuous uniquely ergodic extensions, \cite{Ab-Ka-Le},\cite{Do-Ka}, \cite{Ve}. However, there are topological models of irrational rotations in which  eigenfunctions are not continuous. Indeed, each ergodic transformation admits even a uniquely ergodic model which is topologically mixing \cite{Leh}
.

An important tool for the proof of Theorem~\ref{tw:sa1} is provided by a general criterion for an ergodic automorphism $T$ to satisfy Sarnak's conjecture in each of its uniquely ergodic model.  The criterion is, roughly, to establish an approximative disjointness of sufficiently large relatively prime powers.

\begin{Def}
  We say that the measure-theoretic dynamical system $(X,\mu,T)$ has \emph{Asymptotical Orthogonal Powers} (AOP) if, for any $f$ and $g$ in $L^2(\mu)$ with $\int_X f\, d\mu = \int_X g\, d\mu =0$, we have
\beq\label{math}
\lim_{\stack{r,s\to\infty,}{r,s\text{ different primes}}}  \sup_{\stack{\kappa\text{ ergodic joining}}{\text{of }T^r\text{ and }T^s}}
\left|\int_{X\times X} f\otimes g\, d\kappa \right| = 0.
\eeq
\end{Def}

We would like to emphasize that this AOP property can even be satisfied for an automorphism $T$ for which all non-zero powers are measure-theoretically isomorphic (Example~\ref{ex:one}). Further remarks on the AOP property are discussed in Section~\ref{sec:aop}: we show that it implies zero entropy, but give also examples of zero-entropy automorphisms which do not have AOP.

\begin{Th}\label{tw:sa2}
Assume that $T$ is a totally ergodic automorphism of a standard Borel probability space $\xbm$, and that it has AOP. 
Let $S$ be a uniquely ergodic continuous map of a compact metric space $Y$, and denote by $\nu$ the unique $S$-invariant probability measure. 
Assume that the measure-theoretic systems $(X,\mu,T)$ and $(Y,\nu,S)$ are isomorphic. Then, for each multiplicative function $\bnu:\N\to\C$, $|\bnu|\leq 1$, each $f\in C(Y)$ with $\int_Y f\,d\nu=0$ and each $y\in Y$, we have
\[
\lim_{N\to\infty}\frac1N\sum_{n\leq N}f(S^ny) \bnu(n)=0.
\]
In particular, Sarnak's conjecture holds for $S$.
\end{Th}

Using the fact that, in uniquely ergodic models, the average behavior of points is uniform, we will prove that uniquely ergodic models of automorphisms with AOP enjoy a stronger form of Sarnak's conjecture, in which the sequence output by the system is allowed to switch orbit from time to time.

\begin{Th}\label{tw:sa3} 
Let $(X,T)$ be uniquely ergodic, and let $\mu$ be the unique $T$-invariant probability measure. Assume that, for any uniquely ergodic model 
$(Y,S)$ of $(X,\mu,T)$,  the conclusion of Theorem~\ref{tw:sa2} holds.
Let $1=b_1<b_2<\cdots$ be an increasing sequence of integers with $b_{k+1}-b_k\to\infty$. 
Let $(x_k)$ be an arbitrary sequence of points in $X$. Then,
for each multiplicative function $\bnu:\N\to\C$, $|\bnu|\leq1$ and for each $f\in C(X)$ with $\int_X f\,d\mu=0$, 
we have
\beq\label{rou1}
\lim_{K\to\infty}\frac{1}{b_{K+1}} \sum_{1\le k\le K} \sum_{b_{k}\le n <b_{k+1}}f(T^nx_k)\bnu(n)=0.
\eeq
\end{Th}
Again, if we take $\bnu=\mob$, the conclusion holds even if $\int_X f\,d\mu\neq 0$. 

One of the main results of the paper, Theorem~\ref{thm:AOPQS}, says that all quasi-discrete spectrum automorphisms have AOP. Quasi-discrete spectrum automorphisms (see Section~\ref{sec:qds} for a precise definition) have been first studied by Abramov~\cite{Ab}.
They admit special uniquely ergodic models which are of  the form $Tx=Ax+b$, where $X$ is a compact, Abelian, connected group, $A$ is a continuous group automorphism and $b\in X$ (with some additional assumptions on $A$ and $b$). It has already been proved by Liu and Sarnak in \cite{Li-Sa} that, in these algebraic models, Sarnak's conjecture holds. 
As a consequence of Theorem~\ref{tw:sa2}, we obtain that it holds in \emph{every} uniquely ergodic model of any of these systems (in fact, Theorem~\ref{tw:sa2} directly gives Theorem~\ref{tw:sa1}). 

In Section~\ref{sec:proof-poly}, we apply Theorem~\ref{tw:sa3} in a special situation of such algebraic models which are built by taking affine cocycle extensions of irrational rotations (see \cite{Fu}). This leads to the following theorem.

%

\begin{Th}\label{wn:hmt1}  Assume that $\bnu:\N\to\C$, $|\bnu|\leq1$, is multiplicative. Then, for  each increasing sequence $1=b_1<b_2<\cdots$ with $b_{k+1}-b_k\to\infty$ and for each non constant polynomial $P\in\R[x]$ with irrational leading coefficient, we have
$$
\frac1{b_{K+1}}\sum_{k\leq K}\left|\sum_{b_k\leq n < b_{k+1}}e^{2\pi iP(n)}\bnu(n)\right|\longrightarrow0\text{ when }K\to\infty.$$
\end{Th}

It has been noted to us by N.~Frantzikinakis that the particular case $P(n)=\alpha n$, which we present in Section~\ref{sec:alpha_n},  follows from a recent result of Matom\"aki, Radziwi\l\l\ and Tao \cite{Ma-Ra-Ta}. We can also reformulate Theorem~\ref{wn:hmt1} in the following way.

\begin{Th}
  \label{thm:HX}
  Assume that $\bnu:\N\to\C$, $|\bnu|\leq1$, is multiplicative. For each non constant polynomial $P\in\R[x]$ with irrational leading coefficient, we have
  \[
    \frac{1}{M} \sum_{M\le m<2M} \frac{1}{H} \left| \sum_{m\le n<m+H} e^{2\pi iP(n)}\bnu(n)  \right|\longrightarrow  0
  \]
  as $M\to\infty$, $H\to\infty$, $H/M\to 0$.
\end{Th}


\section{Main tools}

\subsection{Joinings}\label{secjoi} Let $\xbm$ be a standard Borel probability space. We denote by $C_2(X,\mu)$ the corresponding space of {\em couplings}: $\rho\in C_2(X,\mu)$ if $\rho$ is a probability on $X\times X$ with both projections equal to
$\mu$.  The space $C_2(X,\mu)$ is endowed with a topology which coincides with the weak topology when $X$ is a compact metric space. In this topology, convergence of a sequence of couplings $(\rho_n)$ to some coupling $\rho$ is characterized by
\[
  \forall f,g\in L^2(\mu),\ \int_{X\times X} f\otimes g\, d\rho_n \tend{n}{\infty} \int_{X\times X} f\otimes g\, d\rho.
\]
  This topology turns $C_2(X,\mu)$ into a compact metrizable space. For example, the formula 
  \[
    d(\rho_1,\rho_2):=\sum_{k,\ell\geq1}\frac{|\int f_k\ot f_\ell\,d\rho_1-\int f_k\ot f_\ell\,d\rho_2|}{2^{k+\ell}}
  \]
where $\{f_k: k\ge 1\}$ is a linearly $L^2$-dense set of functions of $L^2$-norm~1 yields a metric compatible with the topology on $C_2(X,\mu)$.

Now, let $T$ be a totally ergodic automorphism of $\xbm$, that is,  $T^m$ is assumed to be ergodic for each $m\neq0$.  Let $J(T^r,T^s)$ stand for the set of {\em joinings} between $T^r$ and $T^s$, \textit{i.e.} those $\rho\in C_2(X,\mu)$ which are $T^r\times T^s$-invariant. By
$J^e(T^r,T^s)$ we denote the subset of \emph{ergodic} joinings, which is never empty when both $T^r$ and $T^s$ are ergodic.

\subsection{The KBSZ criterion} We will need the following version of the result of Bourgain, Sarnak and Ziegler \cite{Bo-Sa-Zi} (see also \cite{Ka}, \cite{Ha}).

\begin{Prop} \label{pr:kbsz}Assume that $(a_n)$ is a bounded sequence of complex numbers. Assume, moreover, that
\beq\label{kbsz}
\limsup_{\stack{r,s\to \infty}{\text{ different primes}}}\left(\limsup_{N\to\infty}
\left|\frac1N\sum_{n\leq N}a_{rn}\ov{a}_{sn}\right|\right)=0.\eeq
Then, for each multiplicative function $\bnu:\N\to\C$, $|\bnu|\leq 1$, we have
\beq\label{eq:sa2}
\lim_{N\to\infty}\frac1N\sum_{n\leq N}a_n\cdot \bnu(n)=0.\eeq
\end{Prop}

The proof of this proposition follows directly from \cite{Bo-Sa-Zi} by considering a version of their result omitting a finite set of prime numbers.

\begin{proof}[Proof of Theorem~\ref{tw:sa2}] We begin by observing that the condition of having AOP is an invariant of measure-theoretic isomorphism. Therefore, if it holds for $(X,\mu,T)$, it also holds in each of its uniquely ergodic model. Let us fix such a model $(Y,\nu,S)$.
Take $y\in Y$ and $f\in C(Y)$ with $\int_Yf\,d\nu=0$. Let $\vep>0$. In view of~\eqref{math}, for some $M\geq1$, whenever  $r,s\geq M$ are different prime numbers, we have for each $\kappa\in J^e(S^r,S^s)$
\beq\label{mt0}
\left|\int_{Y\times Y} f\ot\ov{f}\,d\kappa\right|
\leq \vep.\eeq
Now, let $r, s\geq M$ be different prime numbers, and select a subsequence $(N_k)$ so that
\beq\label{mt1}
\frac1{N_k} \sum_{n\leq N_k} \delta_{(S^{rn}y,S^{sn}y)}\to \rho.\eeq
Then, by the unique ergodicity of $S$, which implies the unique ergodicity of $S^r$ and $S^s$, it follows that $\rho\in J(S^r,S^s)$. Let 
$$
\rho=\int_{J^e(S^r,S^s)} \kappa\,dP(\kappa)$$
stand for the ergodic decomposition into ergodic joinings between $S^r$ and $S^s$ of $\rho$ (here, we use again that the non-trivial powers of $S$ are ergodic). Since $f$ is continuous, in view of~\eqref{mt1}, we have
$$
\frac1{N_k} \sum_{n\leq N_k} f(S^{rn}y)\ov{f(S^{sn}y)}\to \int_{Y\times Y}f\ot\ov{f}\,d\rho.$$
But, in view of~\eqref{mt0},
$$
\left|\int_{Y\times Y} f\ot \ov{f}\,d\rho\right|=\left|\int_{J^e(S^r,S^s)} \left(\int_{Y\times Y} f\ot\ov{f}\,d\kappa\right)\,dP(\kappa)\right|\leq\vep.$$
It follows that
$$
\limsup_{N\to \infty}\left|\frac1N\sum_{n\leq N} f(S^{rn}y)\ov{f(S^{sn}y)}\right|\leq \vep.$$
We have obtained that condition~\eqref{kbsz} holds for $a_n:=f(S^ny)$, $n\geq0$, and the result follows from Proposition~\ref{pr:kbsz}.\end{proof}
%
%

\subsection{Special uniquely ergodic models}
\label{clevermodel}Our idea in this section will be, given a uniquely ergodic system $(X,T)$ (with $\mu$ as unique invariant probability measure),  to build many new uniquely ergodic models of the measure-theoretic system $(X,\mu,T)$. Any such model will depend on the choice of an increasing sequence of integers $1=b_1<b_2<\cdots$ with $b_{k+1}-b_k\to\infty$, and of an arbitrary sequence of points $(x_k)$ in $X$.

We define a sequence $\underline{y}=(y_n)\in X^{\N}$ by setting $y_n:=T^n x_k$ whenever $b_k\le n< b_{k+1}$:
\[
  \underline{y}=(Tx_1,\ldots,T^{b_2-1}x_1,
  T^{b_2}x_2,\ldots,T^{b_3-1}x_2,T^{b_3}x_3,\ldots).
\]
Let $S$ denote the shift in $X^{\N}$, where $X^{\N}$ is considered with the product topology, and set
$$
Y:=\ov{\{S^n\underline{y}:\:n\geq0\}}\subset X^{\N}.$$

\begin{Prop}
  The topological dynamical system $(Y,S)$  defined above is uniquely ergodic. 
  The unique $S$-invariant probability measure on $Y$ is
  the graph measure $\nu$ determined  by the formula
$$
\nu(A_1\times A_2\times\ldots\times A_m\times X\times\ldots):=\mu(A_1\cap T^{-1}A_2\cap\ldots\cap T^{-m+1}A_m).
$$
Moreover, the measure-theoretic dynamical system $(Y,\nu,S)$ is isomorphic to $(X,\mu,T)$.
\end{Prop}
%

\begin{proof}
  We first prove that $\underline{y}$ is generic for the measure $\nu$ defined in the statement of the proposition.
  Indeed, let $\kappa$ be a probability measure on $Y$ for which $\underline{y}$ is quasi-generic, \textit{i.e.} assume that there exists a
  sequence of integers $N_j\to\infty$ such that
$$
\frac1{N_j}\sum_{n\leq N_j}\delta_{S^n\underline{y}}\to\kappa.
$$
For fixed $j\ge 1$, the projection of the left-hand side on the first coordinate is
\begin{multline*}
\nu_j:=\frac1{N_j}\left(\sum_{b_1\leq n <  b_2} \delta_{T^nx_1}+
\sum_{b_2\leq n<  b_3} \delta_{T^{n}x_2}+\ldots\right.\\
+ \left.\sum_{b_{k-1}\leq n <   b_k} \delta_{T^n x_{k-1}}+
\sum_{b_k\leq n< b_k+u} \delta_{T^nx_k}\right),
\end{multline*}
where $b_k+u=N_j<b_{k+1}$. Then, the difference the pushforward of $\nu_j$ by $T$ and $\nu_j$  is
\begin{multline*}
T_\ast\nu_j-\nu_j=\frac1{N_j}
\left(
\delta_{T^{b_2}x_1}-\delta_{Tx_1}+
\delta_{T^{b_3}x_2}-\delta_{T^{b_2}x_2}+\ldots+\right.\\
\left.\delta_{T^{b_k}x_{k-1}}-\delta_{T^{b_{k-1}}x_{k-1}}
+\delta_{T^{b_k+u+1}x_k}-\delta_{T^{b_k}x_k}
\right),
\end{multline*} 
and this measure goes weakly to zero as $b_{k+1}-b_k\to\infty$. Hence the limit of $\nu_j$ is $T$-invariant and,
by unique ergodicity, the projection of $\kappa$ on the first coordinate has to
be equal to $\mu$. Now, using again $b_{k+1}-b_k\to\infty$, we see that the proportion of integers
$n\le N_j$ such that the second coordinate of $S^n \underline{y}$ is not the image by $T$ of the first coordinate
of $S^n \underline{y}$ go to zero as $j\to\infty$, and it follows that, $\kappa$-almost surely, 
the second coordinate of $\underline{z}\in Y$ is the image by $T$ of the first coordinate
of $\underline{z}$. The same argument works for any coordinate and we conclude that $\kappa=\nu$.

Now, the fact that each $\underline{z}\in Y$ is also generic for the same measure $\nu$ is a direct consequence of the
following easy fact: 
If $n_j\to\infty$ and $\underline{z}=\lim_{j\to\infty}S^{n_j}\underline{y}$, then either there exists $z\in X$ such that 
$\underline{z}=(z,Tz,T^2z,\ldots)$, or there exist $z_1,z_2\in X$ and $\ell\ge0$ such that
$$
\underline{z}=(z_1,Tz_1,\ldots,T^{\ell}z_1,z_2,Tz_2,T^2z_2,\ldots).$$
Indeed, we can approximate any ``window'' $\underline{z}[1,M]$ by
$S^{n_j}\underline{y}[1,M]=\underline{y}[n_j+1,n_j+M]$, and when $n_j\to\infty$, 
such a window has at most  one point of ``discontinuity'', that is, it contains at most once
two consecutive coordinates which are not successive images by $T$ of some $x_k$.

Finally, by the construction of $\nu$, the map $x\mapsto (x,Tx,T^2x,\ldots)$ is clearly an isomorphism
between $(X,\mu,T)$ and $(Y,\nu,S)$.
\end{proof}

\begin{proof}[Proof of Theorem~\ref{tw:sa3}]
  Apply the conclusion of Theorem~\ref{tw:sa2} to the specific uniquely ergodic model $(Y,S)$ which is constructed
  above, starting from the point $\underline{y}$ and taking an arbitrary continuous function of the first coordinate.
\end{proof}

\section{Asymptotic orthogonality of powers for quasi-discrete spectrum automorphisms}
\subsection{The key lemma}
We say that $c\in \bs^1$ is \emph{irrational} if $c^n\neq1$ for each integer $n\neq0$. Recall that an ergodic automorphism $T$ is
totally ergodic if and only if all its eigenvalues except~1 are irrational.
\begin{Lemma}
  \label{lemma:key}
  Let $c_1$ and $c_2$ be irrational in $\bs^1$. Then there exists at most one pair $(r,s)$ of mutually prime natural numbers
  such that $c_1^r=c_2^s$.
\end{Lemma}

\begin{proof}
Let $(r,s)$ and $(r',s')$ be two pairs of mutually prime natural numbers satisfying $c_1^r=c_2^s$, and $c_1^{r'}= c_2^{r'}$.
  Let $d\in \bs^1$ be such that $d^s=c_1$. We have $(d^r)^s=c_1^r=c_2^s$. Multiplying if necessary $d$ by an appropriate
  $s$-th root of the unity, and using the fact that $r$ and $s$ are mutually prime, we can assume that we also have $d^r=c_2$.
  The equality $c_1^{r'}= c_2^{r'}$ now yields $d^{r's}=d^{s'r}$. But $d$ is also irrational, hence $r's=sr'$. Since $(r,s)=1=(r',s')$, $r=r'$ and $s=s'$.
\end{proof}

\subsection{AOP for discrete spectrum}
Recall that discrete spectrum automorphisms are characterized by the fact that their eigenfunctions form a linearly dense subset of $L^2$. They fall into the more general case of quasi-discrete spectrum ones, but we nevertheless include a 
direct proof that they have AOP as soon as they are totally ergodic, since this will introduce the main ideas in a simpler context. 

Assume that $T$ is a totally ergodic, discrete spectrum automorphism of $\xbm$.
Let $c_1=1, c_2,c_3,\ldots\in \bs^1$ be its eigenvalues, and let $f_1, f_2, \ldots$ be corresponding eigenfunctions: $f_i\circ T=c_i\cdot f_i$, $|f_i|=1$. Note that $c_i$ is irrational and $\int f_i\,d\mu=0$ whenever $i\ge2$.
%
Let us prove that the AOP property holds for $T$. Since the linear subspace spanned by $f_2,f_3,\ldots$ is dense in the subspace of square integrable functions with zero integral, it is enough to show that, for any fixed $i,j\ge 2$,
$$
\int f_i\ot\ov{f_j}\,d\rho=0, \ \forall \rho\in J^e(T^r,T^s),\mbox{ if $r$ and $s$ are large enough and $(r,s)=1$.}$$
But we have
$$
\int f_i\ot\ov{f}_j\,d\rho=\int (f_i\ot\raz_X)\cdot(\raz_X\ot\ov{f}_j)\,d\rho,$$
where $f_i\ot\raz_X$ is an eigenfunction for $(T^r\times T^s,\rho)$ corresponding to the eigenvalue $c_i^{r}$, while
$\raz_X\ot f_j$ is an eigenfunction for $(T^r\times T^s,\rho)$ corresponding to $c_j^s$. 
But, if these two eigenvalues are different, the functions
$f_i\ot \raz_X$ and $\raz_X\ot f_j$ are orthogonal. It is then enough to observe that, by Lemma~\ref{lemma:key},  the equality
$c_i^r=c_j^s$ has at most one solution with $(r,s)=1$.

\begin{Example}
\label{ex:one}
If we consider $T$ an ergodic rotation whose
group of eigenvalues is of the form $\{e^{2\pi iq\alpha}:\: q\in\Q\}$ and $\alpha\in[0,1)$ is irrational, then by the Halmos-von Neumann theorem  (\textit{e.g.}\ \cite{Gl}) it follows that all non-zero powers of $T$ are isomorphic. This is an extremal example of an automorphism with non-disjoint powers 
for which the AOP property holds.\end{Example}

\subsection{Quasi-discrete spectrum}
\label{sec:qds}

We denote by $M=M\xbm$ the multiplicative group of functions $f\in  L^2\xbm$ satisfying $|f|=1$.
Let $T$ be an ergodic automorphism of $\xbm$. We define on $M$ the group homomorphism $W_T$ by setting
\[
  W_T(f) := f\circ T / f.
\]
Note that $(W_Tf)\circ T=W_T (f\circ T)$.

Let $E_0(T)\subset\bs^1$ denote the group of eigenvalues
of $T$, which we also consider as a subgroup of (constant) functions in $M$.
Then, for each integer $k \geq1$, we inductively set
$$
E_k(T) := \{f\in M: W_T f \in E_{k-1}(T)\}.$$

It is easily proved by induction on $k$ that $E_k(T)$ is the subgroup of $M$ constituted of all $f\in M$ satisfying $W_T^k f\in E_0(T)$, and
that, by the ergodicity of $T$,
\[f\in E_k(T)\setminus E_{k-1}(T) \Longleftrightarrow W_T^k f\in E_0(T)\setminus\{1\}.\]

The elements of $\bigcup_{k\geq0}E_k(T)$ are called {\em quasi-eigenfunctions}. Clearly, $E_k(T)\subset E_{k+1}(T)$, $k\geq0$.

We denote by
$E_0^{\ir}(T)$ the set of irrational eigenvalues of $T$, and for each $k\ge1$, we set
\[
  E_k^{\ir} (T) := \{f\in E_k(T): W_T^k f\in E_0^{\ir}(T)\}.
\]

\begin{Lemma}
\label{lemma:integrale_nulle}
  Let $T$ be an ergodic automorphism of $\xbm$, $k\ge1$ and $f\in E_k^\ir(T)$. Then
  \[
    \int_X f\, d\mu = 0.
  \]
\end{Lemma}

\begin{proof}
  We prove the result by induction on $k$. For $k=1$, the conclusion holds for each eigenfunction $f$
  whose eigenvalue is different from~1. Assume that the result is proved for $k\ge1$, and consider $f\in E_{k+1}^\ir(T)$.
  Set $g:=W_Tf\in E_k^\ir(T)$, so that $f\circ T=gf$. Then, for each $n\ge1$, we have
  \[
    f\circ T^n = g^{(n)}\cdot f,
  \]
  where
  \[
    g^{(n)} := g\cdot g\circ T \cdots g\circ T^{n-1}\in E_k(T).
  \]
  It follows that
  \beq
  \label{eq:fourier_of_f}
    \int_X f \cdot \overline{(f\circ T^n)} \, d\mu = \overline{\int_X g^{(n)} \, d\mu}.
  \eeq
  Observe now that
  \[
    W_T^k g^{(n)} = W_T^k g\cdot W_T^k (g\circ T)\cdots W_T^k (g\circ T^{n-1}) = \left( W_T^k g\right)^n = \left( W_T^{k+1} f\right)^n,
  \]
  which is an irrational eigenvalue. We therefore have $g^{(n)}\in E_k^\ir(T)$, and by induction hypothesis,
  \[
    \int_X g^{(n)} \, d\mu = 0.
  \]
  Coming back to~\eqref{eq:fourier_of_f}, we see that the spectral measure of $f$ is the Lebesgue measure on the circle,
  which implies the result.
\end{proof}

\begin{Lemma}
  \label{lemma:Tr}
  Let $T$ be an ergodic automorphism of $\xbm$, $k\ge1$, and $f\in E_k(T)$. Then for each integer $r\ge1$,
  \[
    W_{T^r}^k f = \left( W_T^k f\right)^{r^k}.
  \]
\end{Lemma}

\begin{proof}
  Again, we prove the result by induction on $k$. For $k=1$, take $f\in E_1(T)$ and
  let $c:=W_T f$ be the corresponding eigenvalue (for $T$). Then
  \[
    W_{T^r} f = f\circ T^r / f = c^r =  \left( W_T f\right)^{r}.
  \]
  Now, assume the result is proved for $k\ge1$, and take $f\in E_{k+1}(T)$. Note that $W_{T^r f}\in E_k(T)$,
  so that, using the induction hypothesis, we have
  \[
    W_{T^r}^{k+1} f  = W_{T^r}^{k} \left( W_{T^r f} \right)
     = \left(  W_T^k \,W_{T^r} f \right)^{r^k}
     = \left(  W_T^k f\circ T^r / W_T^k f\right)^{r^k}.
  \]
  But $W_T^k f$ is an eigenfunction of $T$. Denoting by $c:=W_T^{k+1}f$ the corresponding eigenvalue, the right-hand side of the above equality
  becomes $(c^r)^{r^k}=c^{r^{k+1}}$.
  \end{proof}

\begin{Prop}
  \label{prop:at_most_one_rs}
  Let $T$ be a totally ergodic automorphism of $\xbm$, $k\ge1$, $f\in E_k(T)\setminus E_{k-1}(T)$ and $g\in E_k(T)$.
  Then there exists at most one pair $(r,s)$ of mutually prime natural numbers for which we can find $\rho\in J^e(T^r,T^s)$
  satisfying
  \[
    \int_{X\times X} f\otimes \overline g\,d\rho \neq0.
  \]
\end{Prop}

\begin{proof}
Assume that $r$, $s$ and $\rho$ are as in the statement of the proposition.
In the ergodic dynamical system $(X\times X, \rho, T^r\times T^s)$, we have
\[
  W^k_{T^r\times T^s} \left(f\otimes \overline g\right)  = W_{T^r}^k f \otimes W_{T^s}^k \overline{g},
\]
where both $W_{T^r}^k f$ and $W_{T^s}^k \overline{g}$ are eigenvalues of $T$ by Lemma~\ref{lemma:Tr}. This product is therefore
constant, let us denote it by $c$. It is an eigenvalue of $T^r\times T^s$, and it is also the product of two eigenvalues of $T$, hence it is also an eigenvalue of $T$. 
This proves that $f\otimes \overline{g}\in E_k(T^r\times T^s)$.
Since we assume that $\int f\otimes \overline g\, d\rho \neq 0$, Lemma~\ref{lemma:integrale_nulle} yields that
$W^k_{T^r\times T^s} f\otimes \overline g$ is not an irrational eigenvalue of $T^r\times T^s$. But, by the total ergodicity of $T$,
the only eigenvalue of $T$ which is not irrational is~1. Using again Lemma~\ref{lemma:Tr}, we get
\[
  c = 1 = W^k_{T^r\times T^s}\left( f\otimes \overline g\right) = \left(W_T^k f\right)^{r^k} \cdot \left(W_T^k \overline g\right)^{s^k}.
\]
Remember that $W_T^k f$ is irrational because $f\in E_k(T)\setminus E_{k-1}(T)$, and $T$ is totally ergodic. It follows that $W_T^k g$ is also irrational, hence $g\in E_k(T)\setminus E_{k-1}(T)$. Then $W_T^k f$ and $W_T^k \overline g$ are both eigenvalues of $T$ which are different from~1, hence they are both irrational. We only need now to apply Lemma~\ref{lemma:key} to conclude the proof.
\end{proof}

\begin{Def}
  $T$ is said to have {\em quasi-discrete} spectrum if:
\beq\label{m1}
\mbox{$T$ is totally ergodic},
\eeq
and
\beq\label{m2}
L^2\xbm=\ov{\mbox{span}}\left(\bigcup_{k\geq0}E_k(T)\right).\eeq
\end{Def}

\begin{Th}
  \label{thm:AOPQS}
  Let $T$ be an automorphism of $\xbm$ with quasi-discrete spectrum.
  Then it has the AOP property.
\end{Th}

\begin{proof}
  By the definition of quasi-discrete spectrum, it is enough to check that, for any $k\ge1$ and any $f,g\in E_k(T)$, we have,
  for $r,s$ mutually prime and large enough and for each $\rho\in J^e(T^r,T^s)$,
  \[
    \int_{X\times X} f\otimes \overline{g}\, d\rho = 0.
  \]
But this immediately follows from Proposition~\ref{prop:at_most_one_rs}.
\end{proof}

\section{Application to some  algebraic models of discrete and quasi-discrete spectrum transformations} 

\subsection{Irrational rotation}
\label{sec:alpha_n}
Fix some irrational number $\alpha$, and consider the transformation 
$R_\alpha: x\mapsto x+\alpha$ on $\T$. Then $(\T,R_\alpha)$ is uniquely ergodic, and the corresponding measure-theoretic dynamical
system is totally ergodic and has discrete spectrum. It therefore has AOP, and Theorem~\ref{tw:sa3} applies in this case 
(we consider $f(x):=e^{2\pi i x}$).
Then, for any multiplicative function $\bnu$, $|\bnu|\le1$, any sequence $(x_k)$ of points in $\T$ and any increasing sequence 
$1=b_1<b_2<\cdots$ with $b_{k+1}-b_k\to\infty$, we have
\[
\lim_{K\to\infty}\frac{1}{b_{K+1}} \sum_{1\le k\le K} e^{2\pi i x_k} \sum_{b_{k}\le n <b_{k+1}} e^{2\pi i n\alpha}\bnu(n)=0.
\]
Now, for each $k\ge1$, we can choose $x_k\in\T$ such that 
\[
  e^{2\pi i x_k} \sum_{b_{k}\le n <b_{k+1}} e^{2\pi i n\alpha}\bnu(n) = 
  \left| \sum_{b_{k}\le n <b_{k+1}} e^{2\pi i n\alpha}\bnu(n) \right|.
\]
This proves Theorem~\ref{wn:hmt1} in the case of a polynomial of degree~1.

\subsection{Affine  transformations of the $d$-dimensional torus}
\label{sec:proof-poly}
To prove Theorem~\ref{wn:hmt1} in its full form, we generalize the preceding case by considering 
some transformation of the $d$-dimensional torus of the form
\[
  T:(x_1,\ldots,x_d)\longmapsto (x_1+\alpha, x_2+x_1, \ldots, x_{d}+x_{d-1}).
\]
This transformation is an affine transformation, it can be written as $x\mapsto Ax+b$ where $A=[a_{ij}]_{i,j=1}^d$ is the matrix 
defined by $a_1:=1, a_{i-1,i}=a_{ii}:=1$ and all other coefficients equal to zero, and $b:=(\alpha,0,\ldots,0)$.
Taking again $\alpha$ irrational, $(\T^d,T)$ is a uniquely ergodic dynamical system, and it is totally ergodic with respect to the 
Haar measure on $\T^d$, which is the unique invariant measure \cite{Fu}. Moreover, the corresponding measure-theoretic dynamical system
has quasi-discrete spectrum~\cite{Ab}. Hence it has AOP, and Theorem~\ref{tw:sa3} applies again. In particular, we will use this theorem
with the function $f(x_1,\ldots,x_d):=e^{2\pi i x_d}$, observing as in~\cite{Fu,Ei-Wa} that the last coordinate of $T^n(x_1,\ldots,x_d)$
is the following polynomial in $n$:
\[
    {n\choose d}\alpha+{n\choose{d-1}}x_1+\ldots+nx_{d-1}+x_d. 
\]

\begin{proof}[Proof of Theorem~\ref{wn:hmt1}]
Consider a fixed polynomial $P\in\R[x]$, whose leading coefficient is irrational. Then we can always choose $\alpha, x_1,\ldots,x_d$ 
so that 
\[
    {n\choose d}\alpha+{n\choose{d-1}}x_1+\ldots+nx_{d-1}+x_d = P(n) \quad\text{for all $n$}.
\]
Let $1=b_1<b_2<\cdots$ be a sequence of integers satisfying $b_{k+1}-b_k\to\infty$, and define a sequence $(x^k)$ of 
points in $\T^d$ by setting
\[
  x^k := (x_1,x_2,\ldots,x_{d-1}, x_d+t_k),
\]
where $(t_k)$ is a sequence of points in $\T$ to be precised later. Then,  Theorem~\ref{tw:sa3} applied to $T$ 
gives
 \[
 \lim_{K\to\infty}\frac{1}{b_{K+1}} \sum_{1\le k\le K} e^{2\pi i t_k} \sum_{b_{k}\le n <b_{k+1}} e^{2\pi i P(n)}\bnu(n)=0.
\]
Now, for each $k$ we can choose $t_k\in\T$ such that 
\[
  e^{2\pi i t_k} \sum_{b_{k}\le n <b_{k+1}} e^{2\pi i P(n)}\bnu(n)=\left| \sum_{b_{k}\le n <b_{k+1}} e^{2\pi i P(n)}\bnu(n)\right|,
\]
and this concludes the proof.  
\end{proof}

\begin{proof}[Proof of Theorem~\ref{thm:HX}]
  Assume that the result is false. Then we can find  a non constant $P\in\R[x]$ with irrational leading coefficient, 
  and two sequences $(M_\ell)$ and $(H_\ell)$ with
  $H_\ell\to\infty$, $H_\ell/M_\ell\to 0$, satisfying for 
  some $\varepsilon>0$:
  \begin{equation}
    \label{eq:thm_fails}  \forall \ell,\quad 
    \frac{1}{M_\ell} \sum_{M_\ell\le m<2M_\ell} \frac{1}{H_\ell} \left| \sum_{m\le n<m+H_\ell} e^{2\pi iP(n)}\bnu(n)  \right|>\varepsilon.
  \end{equation}
  By  passing to a subsequence if necessary, we can also assume that for each $\ell$, $M_{\ell+1}>2M_\ell+H_\ell$.
  Rewriting the left-hand side of the above inequality as
  \[
    \frac{1}{H_\ell} \sum_{0\le r< H_\ell} \frac{1}{M_\ell/H_\ell} \sum_{\stack{M_\ell\le m<2M_\ell}{m=r \mod H_\ell}} \frac{1}{H_\ell} \left| \sum_{m\le n<m+H_\ell} e^{2\pi iP(n)}\bnu(n)  \right|,
  \]
  we see that for each $\ell$ there exists $0\le r_\ell <H_\ell$ such that
  \beq
  \label{eq:grand}
    \frac{1}{M_\ell/H_\ell} \sum_{\stack{M_\ell\le m<2M_\ell}{m = r_\ell  \mod H_\ell}} \frac{1}{H_\ell} \left| \sum_{m\le n<m+H_\ell} e^{2\pi iP(n)}\bnu(n)  \right| > \varepsilon.
  \eeq
  Now, set 
  \begin{multline*}
    \{1=b_1< b_2 < \ldots \} := \\
\{1\} \cup  \bigcup_{\ell} \{m: M_\ell\le m<2M_\ell+H_\ell\text{ and }m= r_\ell  \mod H_\ell\},
  \end{multline*}
and let $K_\ell$ be the largest $k$ such that $b_k<2M_\ell$. 
Observing that $b_{K_\ell+1}/(2M_\ell)\to1$, and using inequality~\eqref{eq:grand}, we get
\begin{multline*}
 \liminf_{\ell\to\infty} \frac{1}{b_{K_\ell+1}}\sum_{k\leq K_\ell}\left|\sum_{b_k+1\leq n < b_{k+1}}e^{2\pi iP(n)}\bnu(n)\right| \\
 \ge  \liminf_{\ell\to\infty} \frac{1}{2M_\ell}\sum_{\stack{M_\ell\le m<2M_\ell}{m = r_\ell  \mod H_\ell}} \left| \sum_{m\le n<m+H_\ell} e^{2\pi iP(n)}\bnu(n)  \right| 
 \ge  \varepsilon / 2,
\end{multline*}
which contradicts Theorem~\ref{wn:hmt1}.
\end{proof}

\subsection{Algebraic models of general quasi-discrete spectrum transformations}

Assume that $T$ is a quasi-discrete spectrum automorphism of $\xbm$. Then, in view of
\cite{Ab}, \cite{Ha-Pa1}, \cite{Ha-Pa2}, up to measure theoretic isomorphism, we may assume that:
\begin{enumerate}[(i)]
  \item \label{qd1} $X$ is a compact, connected, Abelian group.
  \item \label{qd2}$Tx=Ax+b$, where $A:X\to X$ is a continuous group automorphism, $b\in X$, and the group generated by
$(A-I)X$ and $b$ is dense in $X$.
\item \label{qd3} For each character $\chi\in\widehat{X}$, there exists $m\geq 0$ such that
$\chi\circ (A-I)^m=1$.~\footnote{Note that $\chi(Tx)=\chi(x)\big(\chi((A-I)x)\chi(b)\big)$, so by an easy induction, we obtain that $\chi\in E_m(T)$ if $m$ is the smallest such that $\chi\circ(A-I)^m=1$.}
\end{enumerate}

By \eqref{qd3}, $\bigcap_{n\geq0}(A-I)^nX=\{0\}$, that is, $T$ is unipotent, and by \cite{Ho-Pa1}, \cite{Ho-Pa2}, it follows that $T$ is minimal. Then, in view of \cite{Ha-Pa1}, $(T,X)$ is uniquely ergodic. 
%

\begin{Th} Assume that $T$ satisfies \eqref{qd1},~\eqref{qd2} and~\eqref{qd3}. Let $\bnu:\N\to \C$ be any multiplicative function, with $|\bnu|\le1$. Then, for each $1\neq\chi\in \widehat{X}$ and each increasing sequence $1=b_1<b_2<\ldots$ with $b_{k+1}-b_k\to\infty$ and each choice of $(z_k)\in X$, we have
$$
\frac{1}{b_{K+1}}\sum_{k\leq K}\left|\sum_{b_k+1\leq n < b_{k+1}}\chi(T^{n}z_k)\bnu(n)\right|\longrightarrow 0\text{ when }K\to\infty.$$

\end{Th}
\begin{proof}
We have already noticed that $(X,T)$ is uniquely ergodic. Fix a character $1\neq \chi\in\widehat{X}$. Using~\eqref{qd3}, let $m\geq 0$ be the smallest such that $\chi\circ (A-I)^m=1$. Since $\chi\neq1$, we have $m\geq1$. Let $Y_m:=(A-I)^{m-1}X$. Note that:
\begin{itemize}
\item $Y_m$ is a non-trivial subgroup of $X$ (indeed, otherwise, $\chi\circ(A-I)^{m-1}=1$);
\item $Y_m$ is connected (indeed, $(A-I)^{m-1}$ is a continuous group homomorphism and, by~\eqref{qd1}, $Y_m$ is the image of a connected space);
\item $\chi(Y_m)=\bs^1$ (by the previous observation).\end{itemize}
Note that, for each $y\in Y_m$, we have $\chi\circ A(y)=\chi(y)$, and since $AY_m\subset Y_m$, we also have $\chi\circ A^n(y)=\chi(y)$ for each $n\geq1$. Moreover, for each $y_k\in Y_m$, we have
\beq\label{qd4}
\chi(T^n(z_k+y_k))=\chi(A^ny_k)\chi(T^nz_k)=\chi(y_k)\chi(T^nz_k).\eeq
It follows by the above  that, for each $k\geq1$, we can find $y_k\in Y_m$ such that
\begin{multline*}
\sum_{b_k+1\leq n < b_{k+1}}\chi(T^n(z_k+y_k))\bnu(n)=
\chi(y_k) \sum_{b_k+1\leq n < b_{k+1}}\chi(T^{n}z_k)\bnu(n)\\
=\left|\sum_{b_k+1\leq n < b_{k+1}}\chi(T^{n}z_k)\bnu(n)\right|
\end{multline*}
and the result follows by the AOP property of the system, and Theorem~\ref{tw:sa3}.
\end{proof}

\section{Remarks on the AOP property}
\label{sec:aop}

\begin{Remark}\label{rem:m1} The class of measure-theoretic dynamical systems with AOP is stable under taking factors, since any ergodic joining of the actions of $T^r$ and $T^s$ on a factor $\sigma$-algebra extends to an ergodic joining of $T^r$ and $T^s$. It is also close by inverse limits, since it is enough to check~\eqref{math} for a dense subclass of functions $f$ in $L^2$.\end{Remark}

\begin{Prop}\label{pr:th}
If $T$ satisfies~\eqref{math} then $T$ is of zero entropy.
\end{Prop}
\begin{proof} Because of Remark~\ref{rem:m1}, it is enough to show that there exists a Bernoulli automorphism with arbitrarily small entropy for which~\eqref{math} fails.

Fix $\delta>0$ and let $T$ be the shift on $X=\{0,1\}^{\Z}$ considered with Bernoulli measure $\mu=(\delta,1-\delta)^{\ot\Z}$. Given $r\geq1$, consider the map $\va_r:X\to X$,
$$
\va_r((x_i)_{i\in\Z}):=(x_{ri})_{i\in\Z}$$
for each $(x_i)_{i\in\Z}\in X$. Then, $\va_r\circ T^r=T\circ\va_r$ and $\va_r$ yields a homomorphism (factor map) from $(X,\mu,T^r)$ onto $(X,\mu,T)$ (because the laws of the independent processes $(X_i)_{i\in \Z}$, $X_i((x_j)_{j\in\Z})=x_i$ and $(X_{ri})_{i\in\Z}$ are the same). Moreover, the extension \beq\label{rb}\mbox{$\va_r:(X,\mu,T^r)\to (X,\mu,T)$ is relatively Bernoulli}\eeq
(see \textit{e.g.} \cite{Gl}) since $\va_r^{-1}(\cb)$ has the complementary independent sub-$\sigma$-algebra given by the smallest $T^r$ invariant sub-$\sigma$-algebra making the variables $X_1,\ldots,X_{r-1}$-measurable.

Take now $r\neq s\geq1$ to obtain the joining $\kappa_{r,s}$ defined as
\beq\label{joi}
\int g\ot h\,d\kappa_{r,s}:=\int \E(g|\va_r^{-1}(\cb))(\va_r^{-1}(x))\E(h|\va_s^{-1}(\cb))(\va_s^{-1}(x)) \,d\mu(x).\eeq
In view of~\eqref{rb}, $\kappa_{r,s}$ is ergodic. Now, consider $f\in L^2\xbm$ defined by $f((x_i)_{i\in Z}):=x_0$. We claim that
\beq\label{eq:16}
\E(f|\va_r^{-1}(\cb))(\va_r^{-1}(x))=f(x)=x_0.
\eeq
Indeed, given $x=(x_i)_{i\in\Z}\in X$,
$$\va_{r}^{-1}(x)=\{y\in X:\: y_{ir}=x_i\;\mbox{for each}\;i\in\Z\}$$
and on that fiber we have a relevant Bernoulli measure as the conditional measure. Note however that $f$ is constant on the fiber, in other words, $f$ is measurable with respect to $\va_r^{-1}(\cb)$. Since~\eqref{eq:16} holds also for $\varphi_s$, using additionally~\eqref{joi}, it follows that
$$
\int f\ot \overline{f}\,d\kappa_{r,s}=\|f\|^2_{L^2\xbm}>0,$$
which excludes the possibility of satisfying~\eqref{math}.
\end{proof}

\begin{Remark} There exist some zero-entropy automorphisms for which the AOP property fails. 
Indeed, we will show that it does not hold for horocycle flows.
We have the classical relation between the geodesic and the horocycle flow: for each $s,u\in\R$, we have
$$
g_uh_sg^{-1}_u=h_{e^{-2u}s}.$$
Set $r=e^{-2u}s$. Assume now that $r\neq s$ are coprime natural numbers, $r,s\to \infty$ but we also assume that $r/s\to 1$.
For such pair $(r,s)$, the number $u$ is determined by $e^{-2u}=r/s$, so $u$ is close to zero. We can hence assume that $g_u$, whose graph $\Delta_{g_u}$ is an ergodic joining of $h_r=(h_1)^r$ and $h_s=(h_1)^s$, belongs to the compact set of joinings $\{\Delta_{g_t}:t\in[0,1]\}$. Since a graph joining is clearly never equal to the product measure, there is a positive distance of $\{\Delta_{g_t}:\:t\in[0,1]\}$ from the product measure in the relevant space of couplings. Therefore,~\eqref{math} fails to hold.
\end{Remark}


\noindent
Mariusz Lema\'{n}czyk\\
\textsc{Faculty of Mathematics and Computer Science, Nicolaus Copernicus University, Chopin street 12/18, 87-100 Toru\'{n}, Poland}\par\nopagebreak
\noindent
\textit{E-mail address:} \texttt{mlem@mat.umk.pl}

\medskip

\noindent
Houcein El Abdalaoui, Thierry de la Rue\\
\textsc{Laboratoire de Math\'{e}matiques Rapha\"{e}l Salem, Normandie Universit\'{e}, Universit\'{e} de Rouen, CNRS -- Avenue de l'Universit\'{e} -- 76801 Saint Etienne du Rouvray, France}\par\nopagebreak
\noindent
\textit{E-mail address:} \texttt{elhoucein.elabdalaoui@univ-rouen.fr,\\ Thierry.de-la-Rue@univ-rouen.fr}

\end{document}